\documentclass[reqno]{amsart}
\usepackage{graphicx} 
\usepackage{amsmath}%
\usepackage{amsfonts}%
\usepackage{amssymb}%
\usepackage{graphicx}
\usepackage{mathrsfs}
\usepackage{hyperref}
\usepackage{fullpage}
\usepackage{cite}
\usepackage{mathrsfs}
\usepackage{amsmath}
\usepackage{amssymb}
\usepackage{amsthm}
\usepackage{amsfonts}
\usepackage{amssymb}
\usepackage{amsmath}
\usepackage{enumitem}
\usepackage{mathtools}
\usepackage{mathrsfs}
\usepackage{color}
\usepackage{graphicx}

\def\0{\emptyset}

\def\n{\noindent}

\def\0{\emptyset}

\def\n{\noindent}
\newtheorem{theorem}{Theorem}
\newtheorem{lemma}[theorem]{Lemma}

\newtheorem{claim}[theorem]{Claim}

\newtheorem{definition}[theorem]{Definition}
\newtheorem{observation}[theorem]{Observation}
\newtheorem{problem}[theorem]{Problem}
\newtheorem{remark}[theorem]{Remark}
\usepackage{graphicx}
\usepackage{amsmath}
\usepackage{color}

\begin{document}
\title{Vertex degree sums for rainbow matchings in 3-uniform hypergraphs }
\thanks{Mei Lu was supported by the National Natural Science Foundation of China under Grant No. 12171272.
Yan Wang is supported by National Key R\&D Program of China under Grant No. 2022YFA1006400, National Natural Science Foundation of China under Grant No. 12201400, Explore X project of Shanghai Jiao Tong University and Shanghai Municipal Education Commission No. 2024AIYB003.
Yi Zhang is supported by Fundamental Research Funds for the Central Universities and Innovation Foundation of BUPT for Youth under Grant No. 2023RC49 and National Natural Science Foundation of China under Grant No. 11901048 and 12071002.}

\author{Haorui Liu}
\address{Department of Mathematical
Sciences, Tsinghua University, Beijing, 100084 }
\email{lhr22@mails.tsinghua.edu.cn}

\author{Mei Lu}
\address{Department of Mathematical
Sciences, Tsinghua University, Beijing, 100084}
\email{lumei@tsinghua.edu.cn}

\author{Yan Wang}
\address{School of Mathematical Sciences, CMA-Shanghai, Shanghai Jiao Tong
University, Shanghai 200240, China}
\email{yan.w@sjtu.edu.cn}

\author{Yi Zhang}
\address{ School of Science, Beijing University of Posts and Telecommunications, Beijing, 100876}
\email{shouwangmm@sina.com}
\date{\today}

\begin{abstract}
Let $n \in 3\mathbb{Z}$ be sufficiently large. Zhang, Zhao and Lu proved that if $H$ is a 3-uniform hypergraph with  $n$ vertices  and no isolated vertices,  and if $\deg(u)+\deg(v) > \frac{2}{3}n^2 - \frac{8}{3}n + 2$ for any two vertices $u$ and $v$ that are contained in some edge of $H$, then $ H $ admits a perfect matching.  In this paper, we prove that the rainbow version of  Zhang, Zhao and Lu's result is asymptotically true. More specifically, let $\delta > 0$ and $ F_1, F_2, \dots, F_{n/3} $ be 3-uniform hypergraphs on a common set of $n$ vertices. For each $ i \in [n/3] $, suppose that $F_i$  has no isolated vertices and $\deg_{F_i}(u)+\deg_{F_i}(v) > \left( \frac{2}{3} + \delta \right)n^2$ holds for any two vertices $u$ and $v$ that are contained in some edge of $F_i$.  Then $ \{ F_1, F_2, \dots, F_{n/3} \} $ admits a rainbow matching. Note that this result is asymptotically tight.
\end{abstract}

\maketitle

\section{Introduction}
A $k$-uniform hypergraph (in short, $k$-graph) $H$ is a pair $(V, E)$, where $V = V (H)$ is a finite set of vertices and $E = E(H)$ is a family of $k$-element subsets of $V$. A matching of size $s$ in $H$ is a family of $s$ pairwise disjoint edges of $H$. If a matching covers all the vertices of $H$, then we call it a perfect matching. Given a set $S \subseteq V$, the degree $\deg_H(S)$ of $S$ is the number of edges of $H$ that contain $S$. We simply write $\deg(S)$ when $H$ is obvious from the context. If $S = \{u\}$, then we write $\deg(u)$ instead of $\deg(S)$. Let $\delta_l(H) = \min\{ \deg(S) : S \subseteq V (H), |S| = l\}$.  For $u \in V(H)$, let
$N_H(u) := \{e : e \subseteq V (H) \backslash \{u\}$ and $e \cup \{u\} \in E(H)\}$. When there is no confusion, we also view $N_H(u)$ as a hypergraph with vertex set $V (H) \backslash \{u\}$ and edge set $N_H(u)$.

Given integers $ l < k \leq n$ such that $k$ divides $n$, let $m_l(k, n)$ denote the smallest integer $m$ such that every $k$-graph $H$ on $n$ vertices with $\delta_l(H) \geq m$ contains a perfect matching. In recent years, the problem of determining $m_l(k, n)$ has received much attention (see \cite{Alon,Han,Han3,Kha1,Kha2,Kuhn1,Kuhn2,Mar,Pik,Rod1,Rod2,Rod3,TrZh12, TrZh13, TrZh16}). A well-known result of Ore \cite{Ore} extended Dirac's theorem by determining the smallest degree sum of two non-adjacent vertices that guarantees a Hamilton cycle in graphs. Ore-type results for hypergraphs witnessed much progress several years ago (see \cite{Tang,wang,Yi0,Yi1,Yi2,Yi3,Yi4,zhang,zhang2}).

In a hypergraph, two distinct vertices are \emph{adjacent} if there exists an edge containing both of them. The following are three possible ways of defining the minimum degree sum of a 3-graph $H$.
Let $\sigma_2(H)=\min \{\deg(u)+\deg(v): u, v\in V(H)$ are adjacent$\}$, $\sigma_2'(H) = \min \{\deg(u)+\deg(v): u, v \in V(H)\}$, and $\sigma_2''(H) = \min \{\deg(u)+\deg(v): u, v\in V(H)$ are not adjacent$\}$.

The parameter $\sigma_2'$ is closely related to the Dirac threshold $m_1(3, n)$ -- Zhang Zhao and Lu \cite{zhang} claimed that, following the approach used  in \cite{Kuhn2}, they can prove that when $n$ is divisible by 3 and sufficiently large, every 3-graph $H$ on $n$ vertices with $\sigma_2'(H) \geq  2 (\binom{n-1}{2} - \binom{2n/3}{2})+1$ contains a perfect matching. On the other hand, a condition on $\sigma_2''$ alone does not necessarily ensure the existence of a perfect matching. In fact, consider a 3-graph  $H$ where the edge set consists of all triples containing a fixed vertex. Whatever $\sigma_2''(H)$ is,  $H$ contains no two disjoint edges, as any two vertices in $H$ are adjacent.

Zhang, Zhao and Lu \cite{zhang} determined the minimum degree sum of two adjacent vertices that guarantees a perfect matching in $3$-graphs without isolated vertices. Let us define (potential) extremal $3$-graphs for the matching problem. For $1\le \ell\le 3$, let $H^{\ell}_{n, s}$ denote the 3-graph of order $n$, whose vertex set is divided into two sets $S$ and $T$ of sizes $n- s\ell+1$ and $s\ell -1$, respectively, and whose edge set consists of all triples with at least $\ell$ vertices in $T$. A well-known conjecture of Erd\H{o}s ~\cite{Erd65} states that either $H_{n,s}^1$  or $H_{n,s}^3$  is the densest $3$-graph on $n$ vertices without a matching of size $s$, which is proved by Frankl \cite{Fra}. On the other hand, K\"{u}hn, Osthus and Treglown \cite{Kuhn2} showed that for sufficiently large $n$, $H_{n,s}^1$ has the largest minimum vertex degree among all $3$-graphs on $n$ vertices without a matching of size $s$. For degree sum condition, Zhang Zhao and Lu \cite{zhang} showed the following theorem.

\begin{theorem}\label{theorem1}\cite{zhang}
There exists $n_0 \in \mathbb{N}$ such that the following holds for all integers $n\ge n_0$ that are divisible by $3$. Let $H$ be a $3$-graph of order $n$ without an isolated vertex. If $\sigma_2(H) > \sigma_2(H^2_{n,n/3})= \frac{2}{3}n^2-\frac{8}{3}n+2$, then $H$ contains a perfect matching.
\end{theorem}

Let $\mathscr{F} = \{ F_1, \dots , F_t\}$ be a family of hypergraphs. A set of $t$ pairwise disjoint edges, one from each $F_i$, is called a rainbow matching for $\mathscr{F}$. In this case, we also say that $\mathscr{F}$ or $\{ F_1, \dots , F_t\}$ admits a rainbow matching.

A famous conjecture of Aharoni and Howard generalized Erd\H{o}s matching conjecture to rainbow matchings: Let $t$ be a positive integer and $\mathscr{F} = \{ F_1, \dots , F_t\}$, such that, for $i \in [t]$, $F_i \subseteq \binom{[n]} {k}$ and $e(F_i) > \{ \binom{kt - 1} {k}, \binom{n} {k} - \binom{n - t + 1} {k}\}$; then $\mathscr{F}$ contains a rainbow matching. 
Huang, Loh, and Sudakov \cite{Huang} showed that this conjecture holds for $n > 3k^2t$. Frankl and Kupavskii \cite{Fra1} improved this lower bound to $n \geq 12tk \log(e^2t)$, which was further improved by Lu, Wang and Yu \cite{Lu0} to $n \geq 2kt$. Keevash, Lifshitz, Long and Minzer \cite{Ke,Ke1} independently verified this conjecture for $n > Ckt$ for some (large and unspecified) constant $C$. Kupavskii \cite{Ku} gave the concrete dependencies on the parameters by showing the conjecture holds for $n > 3ekt$ with $t > 10^7$.

Lu, Yu and Yuan \cite{Lu} proved that for a family of $3$-graphs, $\mathscr{F} = \{ F_1, \dots , F_{n/3}\}$ admits a rainbow matching for sufficiently large $n$ if $\delta_1(F_i) > \binom{n-1}{2} - \binom{2n/3}{2}$ for $i\in [n/3]$. Lu, Wang, Yu \cite{Lu1} extended this result to $4$-graphs. It is natural to ask whether the rainbow version of Theorem \ref{theorem1} holds. Therefore, we pose the following problem:

\begin{problem}\label{conjecture1}
For any $\delta > 0$, does there exist $n_0 \in \mathbb{N}$ such that for all integers $n\ge n_0$ divisible by $3$, the following holds?  Let $\mathscr{F} = \{ F_1, \dots , F_{n/3}\}$ be a family of n-vertex 3-graphs without isolated vertex such that $V (F_i) = V (F_1)$ for $i \in [n/3]$. If $\sigma_2(F_i) >  \sigma_2(H^2_{n,n/3})= \frac{2}{3}n^2-\frac{8}{3}n+2$ for $i \in [n/3]$, does $\mathscr{F}$ admit a rainbow matching?
\end{problem}

Note that the family of $n/3$ copies of $H^2_{n,n/3}$ admits no rainbow matchings. The bound on $\sigma_2$ in Problem \ref{conjecture1} is sharp. We prove the following theorem, which shows Problem \ref{conjecture1} is true asymptotically.
\begin{theorem}\label{theorem2}
For any $\delta > 0$, there exists $n_0 \in \mathbb{N}$ such that the following holds for all integers $n\ge n_0$ that are divisible by $3$. Let $\mathscr{F} = \{ F_1, \dots , F_{n/3}\}$ be a family of n-vertex 3-graphs without isolated vertex on the same vertex set for $i \in [n/3]$. If $\sigma_2(F_i) >  (\frac{2}{3} + \delta)n^2$ for $i \in [n/3]$, then $\mathscr{F}$ admits a rainbow matching.
\end{theorem}

To prove Theorem \ref{theorem2}, we convert this rainbow perfect matching problem to a perfect matching problem for a special class of hypergraphs. For any integer $k \geq 2$, a $k$-graph $H$ is $(1, k - 1)$-partite if there exists a partition of $V (H)$ into sets $V_1$, $V_2$ (called partition classes) such that for any $e \in E(H)$, $|e\cap V_1| = 1$ and $|e\cap V_2| = k - 1$. A $(1, k - 1)$-partite $k$-graph with partition classes $V_1$, $V_2$ is {\it balanced} if $(k - 1)|V_1| = |V_2|$. Further, a vertex set $S\subseteq V(H)$ is {\it balanced} if $(k-1)|S\cap V_1|=|S\cap V_2|$.

Let $n \in 3\mathbb{Z}$,  and let $P$ and $Q$ be disjoint sets such that $|P| = n$ and $|Q| = n/3$. Denote $Q = \{u_1, ..., u_{n/3}\}$ and $P = \{v_1, ..., v_{n}\}$. Let $\mathscr{F} = \{F_1, ..., F_{n/3}\}$ be a family of $3$-graphs on the same vertex set $P$. We use $H_{1,3}(\mathscr{F})$ to represent the balanced $(1, 3)$-partite 4-graph with partition
classes $Q$, $P$ and edge set $\bigcup_{i=1}^{n/3}E_i$
 , where $E_i = \{e \cup \{u_i\} : e \in E(F_i)\}$ for $i \in [n/3]$.
If $E(F_i) = E(H^2_{n,n/3})$ and $V (F_i) = V (H^2_{n,n/3})$ for all $i \in [n/3]$, then we write $H^2_{1,3}(n, n/3)$ for $H_{1,3}(\mathscr{F})$.

\begin{observation}\label{observation}
$H_{1,3}(\mathscr{F})$ and $\mathscr{F} = \{F_1, ..., F_{n/3}\}$ satisfy
\begin{enumerate}[label=(\arabic*)]
  \item $E(F_i)$ is the neighborhood of $u_i$ in $H_{1,3}(\mathscr{F})$ for $i \in [n/3]$, and $\mathscr{F}$ admits a rainbow matching if and only if, $H_{1,3}(\mathscr{F})$ has a perfect matching.
\item For any $u_i\in Q$ and $v_j,v_k\in P$, $v_j$ and $v_k$ are adjacent in $F_i$ if and only if $\deg_{H_{1,3}(\mathscr{F})}(\{u_i,v_j,v_k\}) > 0$.
\item For any $u_i\in Q$ and $v_j\in P$, $v_j$ is an isolated vertex in $F_i$ if and only if $v_j$ is not adjacent to $u_i$ in $H_{1,3}(\mathscr{F})$.
\end{enumerate}
\end{observation}

By Observation \ref{observation}, Theorem \ref{theorem2} follows from the following result.

\begin{theorem}\label{theorem3}
For any $\delta > 0$, there exists $n_0 \in \mathbb{N}$ such that for all integers $n\ge n_0$ divisible by $3$ the following holds: Let $H$ be a $(1, 3)$-partite $4$-graph with partition classes $Q$ and $P$, where $|P| = n$ and $|Q| = n/3$. Suppose that every vertex in $ Q$ is adjacent to every vertex in $P$ in $H$. Additionally, for any $u_i\in Q$ and $v_j,v_k\in P$ with $\deg(\{u_i,v_j,v_k\}) > 0$, we have $\deg(\{u_i,v_j\}) + \deg(\{u_i,v_k\}) >  (\frac{2}{3} + \delta)n^2$. Then $H$ admits a perfect matching.
\end{theorem}

The organization of this paper is as follows. In Section $2$, we will find a small {\it absorbing} matching $M'$ in $H$. In Section $3$, we prove that $H\setminus V(M')$ contains a fractional perfect matching. In Section $4$, we show the existence of an almost perfect matching  in $H\setminus V(M')$. Finally, we prove Theorem 5 in Section 5. We use these constants $0 < a\ll  \sigma \ll  \rho' \ll \rho \ll \epsilon \ll \gamma \ll \delta \ll 1$ in the following proof.

\section{Absorbing lemma}

\begin{lemma}
\label{ab}
Let $n$  be  sufficiently large and divisible by $3$. Let H be a $(1, 3)$-partite $4$-graph with partition classes $Q$, $P$ such that $3|Q| = |P|$, and every vertex in $Q$ is adjacent to every vertex in $P$ in $H$.  Let $\rho$, $\rho'$, $\epsilon$ be constants that $0 < \rho' \ll \rho \ll \epsilon \ll 1$. If for any $u_i\in Q$ and $v_j,v_k\in P$ with $\deg(\{u_i,v_j,v_k\}) > 0$,  we have $\deg(\{u_i,v_j\}) + \deg(\{u_i,v_k\}) >  (\frac{1}{2} + \epsilon) n^2$, then $H$ contains a matching $M'$ such that $|M'| < \rho n$ and for any balanced subset $ S \subseteq V(H)\backslash V(M')$ with $|S| < \rho' n $, $H[S\cup V(M')]$ has a perfect matching.
\end{lemma}

\n{\bf Proof of Lemma \ref{ab}.}
Let $Q = \{u_1, ..., u_{n/3}\}$ and $F_i=N_H(u_i)$ for $i\in [n/3]$. For every $A\in V(H)$ with $|A\cap Q| = 1$ and $|A\cap P| = 3$, we call a set $T \in \binom{V(H)} {24}$ an absorbing $24$-set for $A$ if both $H[T]$ and $H[A \cup T]$ contain a perfect matching.

\begin{claim}\label{c}
For every $A\in V(H)$ with $|A\cap Q| = 1$ and $|A\cap P| = 3$, there are at least $\frac{\epsilon^{23/2}}{24!\times 2^{37/2}}n^{24}$ absorbing $24$-sets for $A$.
\end{claim}

\begin{proof}
For any $A\in V(H)$ with $|A\cap Q| = 1$ and $|A\cap P| = 3$, without loss of generality, let $A\cap Q = \{u_1\}$ and $A\cap P = \{v_1, v_2, v_3\}$. For each $i\in [n/3]$, let $x_i\in P$ be a vertex with $\deg_{F_i}(x_i) = \delta_1(F_i)$. Let $B_i :=\{ v\in P \mid \deg_{F_i}(\{v, x_i\}) > 0\}$ be the set of vertices that are adjacent to $x_i$ in $F_i$. We claim that for every $i\in [n/3]$, $|B_i| > \sqrt{2\epsilon}n$.

By assumption and Observation \ref{observation}, $x_i$ is not an isolated vertex in $F_i$, therefore $B_i$ is not empty. Take $y_i\in B_i$. Then $\deg_H(\{u_i,x_i,y_i\}) > 0$. So by assumption, we have $\deg_{F_i}(x_i) + \deg_{F_i}(y_i) >  (\frac{1}{2} + \epsilon) n^2$. Hence
$$
\binom{|B_i|}{2} \geq
\deg_{F_i}(x_i) >
\Big(\frac{1}{2} + \epsilon\Big) n^2 -  \deg_{F_i}(y_i) \geq
\Big(\frac{1}{2} + \epsilon\Big) n^2 - \binom{n}{2} >
\epsilon n^2.
$$
Thus we have $|B_i| > \sqrt{2\epsilon}n$.

Let $C=\{v_j\in P\mid v_j\in B_i$ for at least $\sqrt{2\epsilon}n/6$ different $ i\}$ and $D=P\backslash C$. We have
$$
\frac{\sqrt{2\epsilon}}{3}n^2 <
\sum_{i=1}^{n/3}|B_i| <
\frac{n}{3}|C| + \frac{\sqrt{2\epsilon}n}{6}|D| =
\frac{n}{3}|C| +  \frac{\sqrt{2\epsilon}n}{6}(n-|C|) <
\frac{n}{3}|C| +  \frac{\sqrt{2\epsilon}}{6} n^2.
$$
Then $|C| > \sqrt{2\epsilon}n/2$. Next we show that $|N_H(v) \cap N_H(v')| \geq \epsilon\sqrt{2\epsilon} n^3/6$ for any $v\in C$ and any $v' \in P\setminus \{v\}$. Suppose $v \in B_i$. Then
$$ |N_{F_i}(v)\cap N_{F_i}(v')| \geq
\deg_{F_i}(v) + \deg_{F_i}(v') - \binom{n}{2} \geq
\deg_{F_i}(v) + \deg_{F_i}(x_i) - \frac{1}{2} n^2 \geq
\epsilon n^2.$$
Since $v$ is contained in at least $\sqrt{2\epsilon}n/6$ different $B_i$, we have
$$
|N_H(v) \cap N_H(v')| =
\sum_{i=1}^n |N_{F_i}(v)\cap N_{F_i}(v')|\geq
\sum_{v\in B_i} |N_{F_i}(v)\cap N_{F_i}(v')| \geq \frac{\epsilon\sqrt{2\epsilon}}{6} n^3.
$$

\begin{figure}[!htbp]
\begin{center}
\includegraphics[scale=0.15]{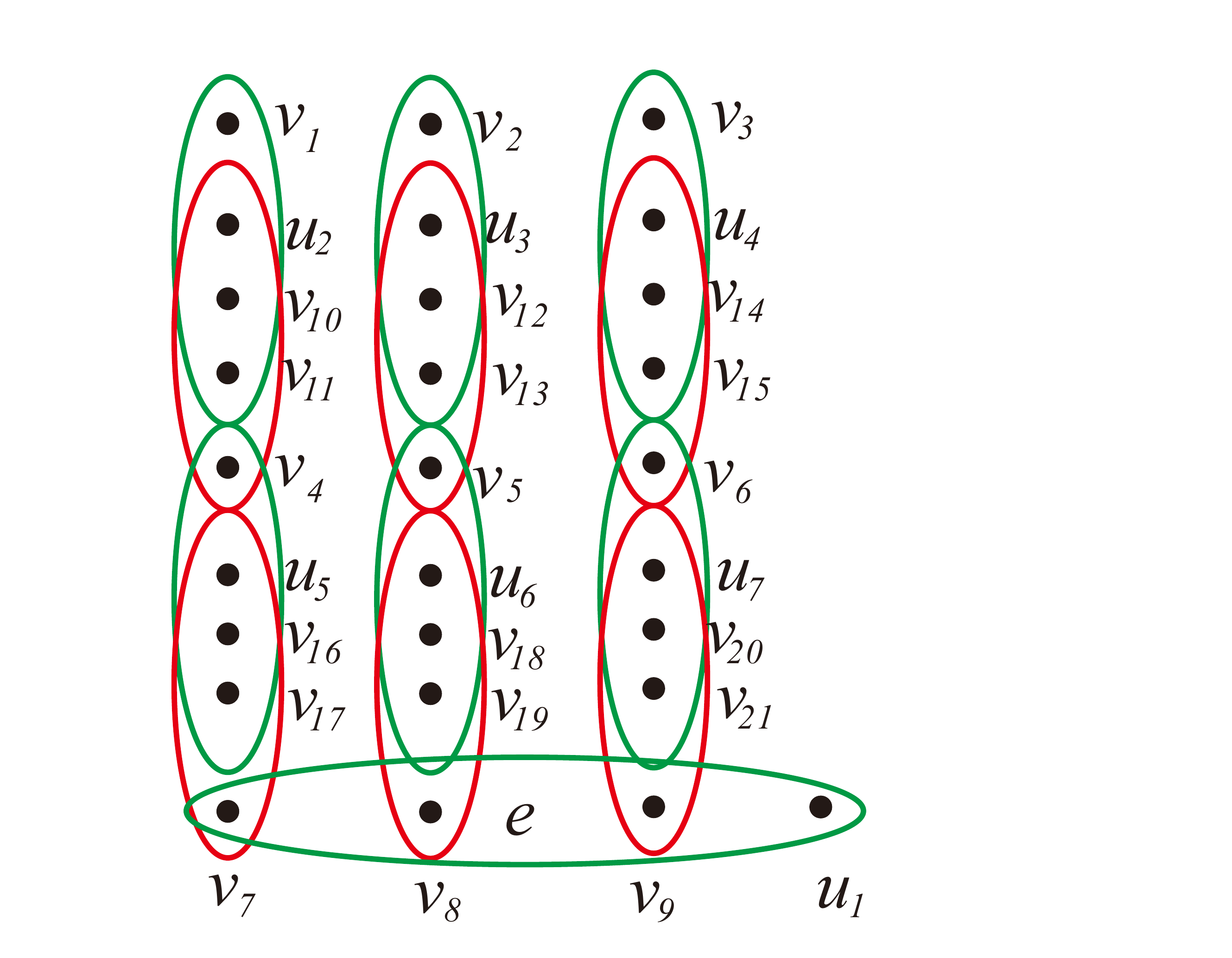}\\
\caption{An absorbing $24$-sets $T$ for $A$ ($T$ is red and$H[A \cup T]$ is green).}\label{Figure1}
\end{center}
\end{figure}

Next we prove that there are at least $\frac{\epsilon^{23/2}}{24!\times2^{37/2}}n^{24}
$ absorbing 24-sets for  $A$. First, we pick $v_4, v_5, v_6 \in C$ different from  $v_1, v_2, v_3$. Since $|C| > \sqrt{2\epsilon}n/2$, we have at least $\sqrt{2\epsilon}n/2-5$ choices for each of $v_4, v_5, v_6$. Second, we pick an  edge $e=\{v_7, v_8, v_9\}$ in $F_1$ disjoint with $\{v_1, v_2,\dots v_6\}$. Since $\delta_1(F_1)\geq \epsilon n^2$, we have $|E(F_1)|\geq \frac{\epsilon}{3} n^3$. Thus we have at least $\frac{\epsilon}{4} n^3$ choices for $e$. Finally, for each $i\in [7]\backslash\{1\}$, we pick $f_i=\{u_i, v_{2i+6}, v_{2i+7}\} \in N_H(v_{i-1})\cap N_H(v_{i+2})$ sequentially disjoint from $\{v_1, v_2,\dots v_{2i+5}\} \cup \{u_1, u_2,\dots u_{i-1}\}$ . Since one of $v_{i-1}$ and $v_{i+2}$ is contained in $\{v_4, v_5, v_6\} \subseteq C$, we have at least $\frac{\epsilon\sqrt{2\epsilon}}{8} n^3$ choices for each $f_i$. Let $T=\{v_4, v_5,\dots v_{21}\} \cup \{u_2, u_3,\dots u_7\}$. Then
$\{f_i\cup \{v_{i+2}\}\mid i\in [7]\backslash\{1\}\}$ is a perfect matching in $T$ and $\{f_i\cup \{v_{i-1}\}\mid i\in [7]\backslash\{1\}\} \cup \{u_1, v_7, v_8, v_9\}$ is a perfect matching in $T\cup A$. Thus $T$ is an absorbing $24$-set for $A$ and the number of choices of $T$ is at least
$$
\Big(\frac{\sqrt{2\epsilon}}{2}n-5\Big)^3\Big(\frac{\epsilon}{4} n^3\Big)\Big(\frac{\epsilon\sqrt{2\epsilon}}{8} n^3\Big)^6\frac{1}{24!}\geq
\frac{\epsilon^{23/2}}{24!\times2^{37/2}}n^{24}.
$$
\end{proof}

We need the following Chernoff bound.
\begin{lemma}\label{cher}\normalfont{[Chernoff bound\cite{Alon}]}
    Suppose $X_1, . . . , X_n$ are independent random variables taking values in $\{0, 1\}$. Let $X$ denote their sum and $\mu = \mathbb{E}[X]$ denote the expected value of $X$. Then for any $0 < \eta < 1$,
    \begin{align*}
    \mathbb{P}[X \geq (1 + \eta)\mu] < e^{-\eta^2\mu/3}
 \  \ \text{and} \ \
\mathbb{P}[X \leq (1 - \eta)\mu] < e^{-\eta^2\mu/2},
\end{align*}
and for any $\eta \geq 1$, $\mathbb{P}[X \geq (1 + \eta)\mu] < e^{-\eta\mu/3}$.
\end{lemma}

Let $\mathscr{L}(A)$ denote the family of all those absorbing $24$-sets of $A$. By Claim \ref{c}, we know $|\mathscr{L}(A)|\geq \frac{\epsilon^{23/2}}{24!\times 2^{37/2}}n^{24}$ for any balanced $4$-set $A$. Let $\mathscr{T}$ be a family of balanced $24$-sets by selecting each of the $\binom{n}{18}\binom{n/3}{6}$ possible balanced $24$-sets independently with probability
\begin{align*}
    p = \frac{\rho n}{2\binom{n}{18}\binom{n/3}{6}}.
\end{align*}

\begin{claim}\label{p}
Let $X$ denote the number of intersecting pairs in $\mathscr{T}$. With positive probability, $\mathscr{T}$ satisfies all the following:
\begin{enumerate}[label=(\arabic*)]
  \item $|\mathscr{T}|\leq \rho n$.
\item $|\mathscr{T}\cap \mathscr{L}(A)| \geq \rho^{1.1}n$.
\item $X\leq \rho^{1.2}n$.
\end{enumerate}
\end{claim}
\begin{proof}
First we observe that  $\mathbb{E}|\mathscr{T}|=\rho n/2$ and $\mathbb{E}|\mathscr{T}\cap \mathscr{L}(A)| = p|\mathscr{L}(A)| \geq 2\rho^{1.1}n$. By Lemma \ref{cher}, with probability $1-o(1)$, both $|\mathscr{T}|\leq \rho n$ and $|\mathscr{T}\cap \mathscr{L}(A)| \geq \rho^{1.1}n$ hold  for all balanced $A \in \binom{V(H)}{4}$.\\
Let $X$ denote the number of intersecting pairs in $\mathscr{T}$. We have $$
\mathbb{E}X= \binom{n/3}{6}\binom{n}{18} \left( 6\binom{n/3-1}{5}\binom{n}{18}+18\binom{n/3}{6}\binom{n-1}{17}\right)p^2 \leq
\frac{\rho^{1.2}n}{2}.
$$
Therefore  by Markov's inequality, we have
$$
\mathbb{P}[X\geq \rho^{1.2}n] \leq
\frac{\mathbb{E}X}{\rho^{1.2}n} \leq
\frac{\rho^{1.2}n/2}{\rho^{1.2}n} =
\frac{1}{2}.
$$
Hence with probability $\frac{1}{2} - o(1)$, $\mathscr{T}$ simultaneously satisfies (1), (2) and (3).
\end{proof}
Let $\mathscr{T}'$ be the collection obtained from $\mathscr{T}$ by removing all intersecting and non-absorbing $24$-sets. Thus  $\mathscr{T}'$ consists of pairwise disjoint absorbing $24$-sets. For any balanced $A \in \binom{V(H)}{4}$, we have the following estimate:
$$
|\mathscr{T}'\cap \mathscr{L}(A)| \geq \rho^{1.1}n-\rho^{1.2}n  \geq \frac{\rho^{1.1}n}{2}.
$$
For any balanced $S\subseteq V(H)$ with $|S| = 4s \leq \rho'n$, we divide $S$ into $s$ balanced $4$-sets, say $\{A_1, A_2, \dots, A_s\}$. Since for each $t \in [s]$, we have  $|\mathscr{T}'\cap \mathscr{L}(A_t)| \geq \rho^{1.1}/2$, and given that $\rho' \ll \rho$, we can greedily pick distinct 24-sets $L_1, L_2, \dots, L_s\in \mathscr{T}'$ such that $L_t \in \mathscr{L}(A_t)$ for all $t \in [s]$. Consequently, $V(M')\cup S$ has a perfect matching.\hfill $\square$

\section{Fractional perfect matching}

Let $H$ be a hypergraph, $p:V(H)\rightarrow [0,1]$ and $q:E(H)\rightarrow [0,1]$. If $\sum\limits_{x\in e}p(x) \geq 1$ for any $e\in E(H)$, then we say that $p$ is a \emph{fractional vertex cover} of $H$.
If $\sum\limits_{x\in e}q(e) \leq 1$ for any $x\in V(H)$, then we say that $q$ is a \emph{ fractional matching} of $H$. We denote the minimum value of a fractional vertex cover of $H$ by:
$$
\tau^*(H)= \min\limits_{p}\sum_{x\in V(H)}p(x),$$
and the maximum value of a fractional matching of $H$ by:
$$
\nu^*(H)= \max\limits_{q}\sum_{e\in E(H)}q(e).$$
By the Strong Duality Theorem of linear programming, we have $\tau^*(H) = \nu^*(H)$. Further, if $\sum\limits_{x\in V(H)}p(x)=\tau^*(H) =n = \nu^*(H) = \sum\limits_{e\in E(H)}q(e)$, we say that $p$ is a \emph{ fractional perfect vertex cover} of $H$ and $q$ is a \emph{ fractional perfect matching} of $H$.

In this section, we will prove the following lemma.

\begin{lemma}\label{f}
There exists $n_0 \in \mathbb{N}$ such that the following holds for all integers $n\ge n_0$ that are divisible by $3$. Let $H$ be a $(1, 3)$-partite $4$-graph with partition classes $Q$, $P$ such that $|P| = n$ and $|Q| = n/3$. Suppose every vertex in $Q$ is adjacent to every vertex in $P$ in $H$. Additionally, assume that for any $u_i\in Q$ and $v_j,v_k\in P$, if $\deg(\{u_i,v_j,v_k\}) > 0$, then $\deg(\{u_i,v_j\}) + \deg(\{u_i,v_k\}) >  \frac{2}{3}n^2-\frac{8}{3}n+2$. Then $H$ has a fractional perfect matching.
\end{lemma}

\begin{remark}
    Lemma \ref{f} is the fractional version of Problem \ref{conjecture1}. In Lemma \ref{f} we only obtain a fractional perfect matching.
\end{remark}

\n{\bf Proof of Lemma \ref{f}}
Let $p$ be a minimum fractional vertex cover of $H$. Let $P= \{v_1, ..., v_{n}\}$, $Q = \{u_1, ..., u_{n/3}\}$ and suppose that $p(v_1)\leq p(v_2)\leq \dots \leq p(v_n)$ and $p(u_1)\leq p(u_2)\leq \dots \leq p(u_{n/3})$. Let $H'$ be the  $(1,3)$-partite $4$-graph with vertex set $V(H)$ and edge set

 \begin{align*}
   E(H')= \left\{ e\in \binom{V(H)}{4} : |e\cap Q |=1
 \ \text{and} \
\sum_{v\in e}p(v)\geq 1\right\}.
\end{align*}
It is not difficult to see that $E(H)\subseteq E(H')$.

\begin{claim}\label{1}
    For any $4$-graph $H''$ with $V(H'')=V(H)$, if $E(H)\subseteq E(H'')\subseteq E(H')$, then $\nu^*(H)= \nu^*(H'')$. In particular, $H''$ has a fractional perfect matching if and only if $H$ has a fractional perfect matching.
\end{claim}

\begin{proof}
Let $p''$ be a minimum fractional vertex cover of $H''$. Since $p$ is a fractional vertex cover of $H'$, $p$ is also a fractional vertex cover of $H''$. Then $\sum\limits_{x\in V(H'')}p(x)\geq \sum\limits_{x\in V(H'')}p''(x)$. Since $V(H)=V(H'')$, we obtain  $\sum\limits_{x\in V(H)}p(x)\geq \sum\limits_{x\in V(H)}p''(x)$. Since $p$ is a minimum fractional vertex cover of $H$ and $p''$ is also a fractional vertex cover of $H$, we have $\sum\limits_{x\in V(H)}p(x)\leq \sum\limits_{x\in V(H)}p''(x)$. Hence, $\sum\limits_{x\in V(H)}p(x)=\sum\limits_{x\in V(H'')}p''(x)$ and it follows that $\tau^*(H)=\tau^*(H'')$. Then $\nu^*(H)= \nu^*(H'')$.
\end{proof}

\begin{definition}
Let $H_1$ be a $(1,3)$-partite $4$-graph with partition classes $Q = \{u_1, u_2, \dots, u_{n/3}\}$ and $P = \{v_1, v_2, \dots, v_n\}$. For $e=\{u_i, v_{j_1}, v_{j_2}, v_{j_3}\}\in \binom{Q}{1}\times\binom{P}{3}$ and $e'=\{u_{i'}, v_{j_1'}, v_{j_2'}, v_{j_3'}\} \in \binom{Q}{1}\times\binom{P}{3}$, we say $e \prec e'$ if $1 \leq i \leq i' \leq n/3$ and $1 \leq j_k \leq j_k' \leq n$ for all $k \in [3]$.
\end{definition}

\begin{definition}
For a $(1,3)$-partite $4$-graph $H_1$ with partition classes $Q = \{u_1, u_2, \dots, u_{n/3}\}$ and $P = \{v_1, v_2, \dots, v_n\}$, we say that $H_1$ is \textit{stable} if for any $e, e' \in \binom{Q}{1}\times\binom{P}{3}$ satisfying $e \prec e'$, we have:
\[
e \in E(H_1) \implies e' \in E(H_1).
\]
\end{definition}

Note that $H'$ is stable with respect to the given labeling of $Q$ and $P$.

\begin{claim}\label{2}
There exists a $(1,3)$-partite $4$-graph $H''$ with vertex set $V(H)$ satisfying:
(1) $E(H) \subseteq E(H'') \subseteq E(H')$; (2) $H''$ is stable with respect to the given labeling; (3) For any $u_i \in Q$ and $v_j, v_k \in P$ with $\deg_{H''}(\{u_i,v_j,v_k\}) > 0$, we have:
\[
\deg_{H''}(\{u_i,v_j\}) + \deg_{H''}(\{u_i,v_k\}) > \frac{2}{3}n^2 - \frac{8}{3}n + 2.
\]
\end{claim}

\begin{proof}
We construct $H''$ by an iterative process. Let $H_0 = H'$ and for $t \geq 0$, we proceed as follows:

\begin{enumerate}
\item Check if the following condition holds for all $u_i \in Q$ and $v_j, v_k \in P$ with $\deg_{H_t}(\{u_i,v_j,v_k\}) > 0$:
\[
\deg_{H_t}(\{u_i,v_j\}) + \deg_{H_t}(\{u_i,v_k\}) > \frac{2}{3}n^2 - \frac{8}{3}n + 2.
\]
If yes, set $H'' = H_t$ and stop.

\item Otherwise, choose $u_{i_t} \in Q$ and $v_{j_t}, v_{k_t} \in P$ that: (i) Satisfy $\deg_{H_t}(\{u_{i_t},v_{j_t},v_{k_t}\}) > 0$; (ii) Have $\deg_{H_t}(\{u_{i_t},v_{j_t}\}) + \deg_{H_t}(\{u_{i_t},v_{k_t}\}) \leq \frac{2}{3}n^2 - \frac{8}{3}n + 2$;
(iii) Minimize $i_t + j_t + k_t$ among all such triples.
Define $H_{t+1}$ by $V(H_{t+1}) = V(H)$ and $E(H_{t+1}) = E(H_t) \setminus \{e \in E(H_t) : \{u_{i_t}, v_{j_t}, v_{k_t}\} \subseteq e\}$.

\end{enumerate}

The algorithm ends in finite steps, since we remove at least one edge in each iteration. We prove by induction that for all $t$,  $H_t$ is stable and $E(H) \subseteq E(H_t)$. This is true for $t = 0$ since $H_0 = H'$. Assume that the statement holds for $t\geq 0$. For $t+1$,  first we show $E(H) \subseteq E(H_{t+1})$. Note that:
\[
\begin{aligned}
\deg_H(\{u_{i_t},v_{j_t}\}) + \deg_H(\{u_{i_t},v_{k_t}\}) &\leq \deg_{H_t}(\{u_{i_t},v_{j_t}\}) + \deg_{H_t}(\{u_{i_t},v_{k_t}\})\leq \frac{2}{3}n^2 - \frac{8}{3}n + 2.
\end{aligned}
\]
Thus, $\deg_H(\{u_{i_t},v_{j_t},v_{k_t}\}) = 0$, meaning no edge containing $\{u_{i_t}, v_{j_t}, v_{k_t}\}$ belongs to $E(H)$.

Secondly, we prove that $H_{t+1}$ is stable. Suppose not, then there exists two edges $e \notin E(H_{t+1})$ and $e' \in E(H_{t+1})$ such that $e' \prec e$. Since $E(H_{t+1}) \subset E(H_{t})$ and $H_t$ is stable, we have $e \in E(H_{t})\backslash E(H_{t+1})$ and $e' \notin E(H_{t})\backslash E(H_{t+1})$. By Step (2) of the construction process, we have $\{u_{i_t}, v_{j_t}, v_{k_t}\} \subset e$ and $\{u_{i_t}, v_{j_t}, v_{k_t}\} \not\subset e'$. Since we have $e' \prec e$, there exist $i_t' \leq i_t$, $j_t' \leq j_t$, and $k_t' \leq k_t$ with $i_t' + j_t' + k_t' < i_t + j_t + k_t$ such that $\{u_{i_t'}, v_{j_t'}, v_{k_t'}\} \subset e'$.
Since $H_t$ is stable, we have $\deg_{H_t}(\{u_{i_t'},v_{j_t'}\}) \leq \deg_{H_t}(\{u_{i_t},v_{j_t}\})$ and
$\deg_{H_t}(\{u_{i_t'},v_{k_t'}\}) \leq \deg_{H_t}(\{u_{i_t},v_{k_t}\}),
$
which implies that \[
\deg_{H_t}(\{u_{i_t'},v_{j_t'}\}) + \deg_{H_t}(\{u_{i_t'},v_{k_t'}\}) \leq \frac{2}{3}n^2 - \frac{8}{3}n + 2.
\]
Also, since $\{u_{i_t'},v_{j_t'},v_{k_t'}\} \subset e'$, we have \[
\deg_{H_t}(\{u_{i_t'},v_{j_t'},v_{k_t'}\}) > 0.
\]
But we have $i_t' + j_t' + k_t' < i_t + j_t + k_t$, which leads to a contradiction to our choice of $\{u_{i_t}, v_{j_t}, v_{k_t}\}$ in step (2). This contradiction establishes that $H_{t+1}$  is still stable, which completes  the proof.
\end{proof}


By Claim \ref{1}, it suffices to prove that the hypergraph $H''$ constructed in Claim \ref{2} has a fractional perfect matching. Let $F' = N_{H''}(u_1)$ be the neighborhood of vertex $u_1$ in $H''$. For any pair of vertices $v_j, v_k \in P$ that are adjacent in $F'$, we have:
    \[
        \deg_{F'}(v_j) + \deg_{F'}(v_k) > \frac{2}{3}n^2 - \frac{8}{3}n + 2.
    \]
By Theorem \ref{theorem1}, this degree condition guarantees that $F'$ contains a perfect matching. Let us denote this perfect matching as: $\{e_1, e_2, \dots, e_{n/3}\}$.
Since $H''$ is stable, we can extend this perfect matching to $H''$ by adding vertices from $Q$ in order:
        $\{e_1 \cup \{u_1\}, e_2 \cup \{u_2\}, \dots, e_{n/3} \cup \{u_{n/3}\}\}$. This forms a perfect matching in $H''$.  Therefore: $ \nu^*(H) = \nu^*(H'') = n/3$, which proves that $H$ contains a fractional perfect matching.
\hfill $\square$

\section{Almost perfect matching}
Similar to \cite{Lu}, we will find an almost perfect matching in $H$ using  probabilistic methods. The following lemma (Lemma 5.2 in \cite{Lu}) will play a key role in our analysis.

\begin{lemma}\label{41}\cite{Lu}
Let $n$ be a sufficiently large positive integer and let $H$ be a $(1, 3)$-partite $4$-graph with partition classes $Q$, $P$ such that $3|Q| = |P| = n$. Let $S \subseteq V (H)$ be a set of vertices such that $|S \cap Q| = n^{0.99}/3$ and $|S \cap P| = n^{0.99}$. Take $n^{1.1}$ independent copies of
$R_+$ and denote them by $R^i_
+$, $1 \leq i \leq n^{1.1}$, where $R_+$ is chosen from $V (H)$ by taking each
vertex uniformly at random with probability $n^{-0.9}$. Define $R^i_- = R^i_+ \backslash S$ for $1 \leq i \leq n^{1.1}$.

Then, with probability $1 - o(1)$, for any sequence $R^i$, $1 \leq i \leq n^{1.1}$, satisfying $R^i_- \subseteq R^i \subseteq R^i_+$, all of the following hold:
\begin{enumerate}[label=(\arabic*)]
  \item $|R^i| = (4/3 + o(1))n^{0.1}$ for all $i = 1, ..., n^{1.1}$.
\item For each $X \subseteq V (H)$, let $Y_X := |\{i : X \subseteq R^i
\}|$, then,
 \begin{enumerate}[label=(\alph*)]
 \item $Y_{\{v\}} \leq (1 + o(1))n^{0.2}$ for $v \in V (H)$,
\item $Y_{\{v\}} = (1 + o(1))n^{0.2}$ for $v \in V (H) \backslash S$,
\item $Y_{\{u,v\}} \leq 2$ for distinct $u, v \in V (H)$, and
\item $Y_e \leq 1$ for $e \in E(H)$.
 \end{enumerate}
\end{enumerate}
\end{lemma}

\begin{lemma}\label{42}
Let $n$, $H$, $P$, $Q$, $S$ and $R^i_+$, $R^i_-$, $i \in [n^{1.1}]$, be given as in Lemma \ref{41} and  $u$ and $v$ are adjacent in $H$ for any $u\in Q$, $v\in P$. For each $X \in \binom{V (H)}{2}$, let $DEG^i_X = |N_H(X) \cap \binom{R^i}{2}|$. Then with probability $1 - o(1)$, for any sequence $R^i$, $1 \leq i \leq n^{1.1}$, satisfying $R^i_- \subseteq R^i \subseteq R^i_+$, the following holds:
If $\delta > 0$ is a constant and $d_H(\{u, v\}) + d_H(\{u, v'\}) \geq (\frac{2}{3} + \delta)n^2$ for all $u \in Q$ and $v, v' \in P$ with $\deg_H(\{u, v, v'\}) >0$, then with probability $1 - o(1)$, we have
\begin{enumerate}[label=(\alph*)]
 \item For all $i \in [n^{1.1}]$, $u\in Q$ and $v\in P$, $DEG^i_{\{u,v\}} >0$.
\item For all $i \in [n^{1.1}]$, $u \in Q$ and $v, v' \in P$ with $\deg_{H[R^i]}(\{u, v, v'\}) > 0$,
$$
 DEG^i_{\{u,v\}} + DEG^i_{\{u,v'\}} \geq \frac{2}{3}|R^i\cap P|^2-\frac{8}{3}|R^i\cap P|+2.
 $$

 \end{enumerate}

\end{lemma}

\begin{proof}
We have
$$
\mathbb{E}(|R^i_-\cap Q|) = \Big(\frac{n}{3} -\frac{n^{0.99}}{3}\Big)n^{-0.9}= \frac{n^{0.1}}{3} -\frac{n^{0.09}}{3}
$$
and
$$
\mathbb{E}(|R^i_-\cap P|) = (n -n^{0.99})n^{-0.9}= n^{0.1} -n^{0.09}.
$$
By Lemma \ref{cher}, we have
$$
\mathbb{P}\Big(|R^i_-\cap Q|-\Big(\frac{n^{0.1}}{3}-\frac{n^{0.09}}{3}\Big)\geq n^{0.095}\Big)\leq e^{-\Omega(n^{0.09})}
$$
and
$$
\mathbb{P}\Big(|R^i_-\cap P|-(n^{0.1}-n^{0.09})\geq n^{0.095}\Big)\leq e^{-\Omega(n^{0.09})}.
$$
     Let $\deg^i_X = |N_H(X) \cap \binom{R^i_-}{2}|$.
    For any $u\in Q$ and $v\in P$, since $d_H(\{u, v\}) \geq (\frac{2}{3} + \delta)n^2 -\binom{n}{2} = \Omega(n^2)$, we have
    $$
    \mu :=
    \mathbb{E}(\deg^i_{\{u, v\}}) =
    d_{H-S}(\{u,v\})(n^{-0.9})^2 \geq
    (1 - o(1))d_H(\{u,v\})(n^{-0.9})^2 =
    \Omega(n^{0.2}).
    $$
    We use Janson's inequality (Theorem 8.7.2 in \cite{Alon2}) to bound the deviation of $\deg^i_{\{u, v\}}$. We write $\deg^i_{\{u, v\}} = \sum_{e\in N_H(\{u,v\})}X_e$, where $X_e=1$ if $e\subseteq R^i_-$ and $X_e=0$ otherwise. Then
    $$
    \Delta=
    \sum_{e\cap f \neq \emptyset}\mathbb{P}(X_e=X_f=1)\leq
    \binom{n-1}{2}\binom{2}{1}\binom{n-3}{1}(n^{-0.9})^3 =
    \Omega(n^{0.3}).
    $$
    By Janson's inequality, for any $\gamma>0$, we have
    $$
    \mathbb{P}\Big(\deg^i_{\{u, v\}}\leq (1-\gamma)\mathbb{E}(\deg^i_{\{u, v\}})\Big) \leq
    e^{\frac{-\gamma^2\mu}{2+\Delta/\mu}}\leq
    e^{-\Omega(n^{0.1})}.
    $$
    With probability at least $1 - n^{2+1.1}e^{-\Omega(n^{0.1})}$, for all $u\in Q$, $v\in P$ and all $i\in [n^{1.1}]$, we have
    $$
    \deg^i_{\{u, v\}} >
    (1-\gamma)\mathbb{E}(\deg^i_{\{u, v\}}) \geq
    (1-\gamma)(1 - o(1))d_H(\{u,v\})(n^{-0.9})^2.
    $$
    Thus, with probability $1-o(1)$, for all $i\in [n^{1.1}]$, $u \in Q$ and $v \in P$, we have
    $$
    DEG^i_{\{u,v\}} \geq
    \deg^i_{\{u,v\}} >0,
    $$
    and for all $i\in [n^{1.1}]$, $u \in Q$ and $v, v' \in P$ with $\deg_{H[R^i]}(\{u, v, v'\}) >0$, we have
    \begin{align*}
DEG^i_{\{u,v\}} + DEG^i_{\{u,v'\}}  & \geq deg^i_{\{u,v\}} + \deg^i_{\{u,v'\}} \\
& \geq (1-\gamma)(1 - o(1))n^{-1.8}(d_H(\{u,v\})+d_H(\{u,v'\})\\
&  \geq \frac{2}{3}(1-\gamma)(1 - o(1))\Big(1+\frac{3}{2}\delta\Big)n^{0.2} \\
&       \geq  \frac{2}{3}|R^i\cap P|^2-\frac{8}{3}|R^i\cap P|+2,
\end{align*}
where the third inequality holds because $\deg_{H[R]}(\{u, v, v'\}) >0$, which follows from  $\deg_{H[R^i]}(\{u, v, v'\}) >0$, and the fourth inequality holds since $\gamma \ll \delta$.
\end{proof}

\begin{lemma}\label{43}(Lemma 5.3 in \cite{Lu})
    Let $n$, $H$, $P$, $Q$, $S$ and $R^i_+$, $R^i_-$, $i \in [n^{1.1}]$, be given as in Lemma \ref{41}. Then with probability $1-o(1)$, for every $i\in [n^{1.1}]$, there exist subgraphs $R^i$ such that $R^i_- \subseteq R^i \subseteq R^i_+$ and $R^i$ is balanced.
\end{lemma}

\begin{lemma}\label{44}(Lemma 5.4 in \cite{Lu})
    Let $n$, $H$, $P$, $Q$, $S$ and $R^i$, $i \in [n^{1.1}]$, be given as in Lemma \ref{43} such that each $H[R^i]$ is a balanced $(1, 3)$-partite $4$-graph and has a fractional perfect matching $w_i$. Then there exists a spanning subgraph $H''$ of $H' := \cup_{i=1}^{n^{1.1}} H[R^i]$ such that
    \begin{enumerate}[label=(\arabic*)]
        \item $\deg_{H''} (u) \leq (1 + o(1))n^{0.2}$ for $u \in S$,
        \item $\deg_{H''} (u) = (1 + o(1))n^{0.2}$ for $u \in V(H)\backslash S$,
        \item $\Delta_2(H'') \leq n^{0.1}$.
    \end{enumerate}
\end{lemma}

The following lemma is due to Pippenger and Spencer \cite{Pi}, stated as Theorem 4.7.1 in \cite{Alon2}. A cover in a hypergraph $H$ is a set of edges whose union is $V(H)$.
\begin{lemma}\label{45}
    (Pippenger and Spencer, 1989) For every integer $k \geq 2$ and reals $r \geq 1$ and
$a > 0$, there are $\tau = \tau(k, r, a) > 0$ and $d_0 = d_0(k, r, a)$ such that for every $n$ and $D \geq d_0$ the following holds: Every $k$-uniform hypergraph $H = (V, E)$ on a set $V$ of $n$ vertices in which all vertices have positive degrees and which satisfies the following conditions:
\begin{enumerate}[label=(\arabic*)]
        \item For all vertices $x \in V$ but at most $\tau n$ of them, $\deg(x) = (1 \pm \tau)D$,
        \item For all $x \in V$ , $\deg(x) < rD$,
        \item For any two distinct $x, y \in V$, $\deg(\{x, y\}) < \tau D$,
    \end{enumerate}
    contains a cover of at most $(1 + a)(n/k)$ edges.
\end{lemma}
Note that $H$ contains a cover of at most $(1 + a)(n/k)$ edges, then $H$  must contain a matching of size at least $(1 - (k - 1)a)(n/k)$. Now we are ready
to state and prove the main result of this section, which will be used to find an almost perfect matching after deleting an absorber.

\begin{lemma}\label{46}
    For any $0 < \sigma \ll \delta $, there exists $n_0 \in \mathbb{N}$ such that the following holds for all integers $n\ge n_0$ that are divisible by $3$. Let $H$ be a $(1, 3)$-partite $4$-graph with partition classes $Q$, $P$ such that $|P| = n$ and $|Q| = n/3$ and  every vertex in $ Q$ is adjacent to every vertex in $P$ in $H$. If for any $u_i\in Q$ and $v_j,v_k\in P$ with $\deg(\{u_i,v_j,v_k\}) > 0$,  $\deg(\{u_i,v_j\}) + \deg(\{u_i,v_k\}) \geq  (2/3+\delta/2)n^2$, then $H$ has a matching covering all but at most $\sigma n$ vertices.
\end{lemma}

\begin{proof}
    Let $\{R^i\mid i\in [n^{1.1}]\}$ be subgraphs of $H$ satisfying Lemmas \ref{41}, \ref{42} and \ref{43}. Then for any $u \in R^i \cap Q$ and $v \in R^i \cap P$, we have
    $ \deg_{H[R^i]}(\{u,v\})= DEG^i_{\{u,v\}} >0.$ Additionally, for all $u \in R^i \cap Q$ and $v,v' \in R^i \cap P$ with  $ \deg_{H[R^i]}(\{u,v,v'\}) >0$, the following holds:
    $$\deg_{H[R^i]}(\{u,v\})+\deg_{H[R^i]}(\{u,v'\})=
    DEG^i_{\{u,v\}}+DEG^i_{\{u,v'\}} \geq
    \frac{2}{3}|R^i\cap P|^2-\frac{8}{3}|R^i\cap P|+2.
    $$
    Therefore, by Lemma \ref{f}, each $R^i$ has a fractional perfect matching for all $i\in [n^{1.1}]$.

    By Lemma \ref{44}, there exists a spanning hypergraph $H''$ of $\bigcup\limits_{i=1}^{n^{1.1}} H[R^i]$ such that $d_{H''}(u) \leq (1+o(1))n^{0.2}$ for all $u \in S$, $d_{H''}(v) = (1+o(1))n^{0.2}$ for all $v \in V(H'')\backslash S$ and $\Delta_2(H'')\leq n^{0.1}$. Hence, by Lemma \ref{45} (with $D= n^{0.2}$), $H''$ contains a edge cover of size at most $(1+a)n/3$, where $0< a \ll \sigma$.

    Thus we greedily delete intersecting edges in the edge cover and obtain a matching of size at least $(1-3a)n/3$.
\end{proof}

\section{Proof of Theorem \ref{theorem2}}
Let $H=H_{1,3}(\mathscr{F})$. By Lemma \ref{ab}, $H_{1,3}(\mathscr{F})$ has a matching $M'$ that $|M'| < \rho n$, and for any balanced $S \subseteq V(H)$ with $|S| \leq \rho' n$, $H[S\cup V(M')]$ has a perfect matching. Let $H_1$ be the hypergraph induced by $V(H)\backslash V(M')$. We have the following inequality for any $u\in V(H_1)\cap Q$ and $v, v' \in V(H_1)\cap P$ with $\deg_{H_1}(\{u,v,v'\})>0$:  $$\deg_{H_1}(u,v) +\deg_{H_1}(u,v') \geq \Big(\frac{2}{3}+\frac{\delta}{2}\Big)n'^2,$$ where $n':=|V(H_1)|=(1-o(1))n$ and $ \delta \gg \rho$. For any $u\in V(H_1)\cap Q $ and $v \in V(H_1)\cap P$, we also have $$\deg_{H_1}(u,v) \geq \Big(\frac{2}{3}+\frac{\delta }{2}\Big)n'^2-\binom{n'}{2} \geq \frac{1}{6}n'^2,$$ it follows that  every vertex of $V(H_1)\cap Q$ is adjacent to every vertex of $V(H_1)\cap P$ in $H_1$. By Lemma \ref{46}, $H_1$ contains a matching $M_1$ covering all but at most $\sigma n$ vertices. Since  $\sigma \ll \rho'$, the matching $M'$ absorbs the set $V(H_1)\backslash V(M_1)$, forming a matching $M_2$. Therefore $M_1 \cup M_2$ is a perfect matching of $H$.

\end{document}